\documentclass[11pt]{amsart}

\usepackage{xypic}
\xyoption{all}
\usepackage{amstext}
\usepackage{amsthm}
\usepackage{amsopn}
\usepackage{amsfonts}
\usepackage{amsmath}
\usepackage{latexsym}
\usepackage{amscd}
\usepackage{amssymb}
\usepackage{amsmath}

\swapnumbers
\newtheorem{theorem}[subsection]{Theorem}

\newtheorem{proposition}[subsection]{Proposition}
\newtheorem{lemma}[subsection]{Lemma}
\newtheorem{corollary}[subsection]{Corollary}

\theoremstyle{definition}

\newtheorem{remark}[subsection]{Remark}
\newtheorem{remarks}[subsection]{Remarks}
\newtheorem{example}[subsection]{Example}

\newtheorem{conjecture}[subsection]{Conjecture}

\numberwithin{equation}{subsection}

\def\gg{\mathfrak g}

\def\GG{\mathfrak G}

\def\HH{\mathfrak H}

\def\TT{\mathfrak T}

\def\G{\mathbf G}

\def\et{\acute et}

\newcommand{\Aut}{\operatorname{Aut}}
 
\newcommand{\Spec}{\operatorname{Spec}}

\newcommand{\GL}{{\operatorname{GL}}}

\newcommand{\simlgr}{\buildrel \sim \over \lra}
\newcommand{\PGL}{{\operatorname{PGL}}}

\newcommand{\lra}{\longrightarrow}

\newcommand{\Br}{\operatorname{Br}}

\newcommand{\cal}{\mathcal}

\newcommand{\calA}{\mathcal A}

\newcommand{\calL}{\mathcal L}

\newcommand{\calF}{\mathcal F}

\def\ol{\overline}
\def\Z{\mathbb Z}

\def\bF{\text{\rm \bf F}}

\def\bG{\text{\rm \bf G}}

\def\cL{\mathcal{L}}

\date{\today}

\begin{document}

\title[Torsors over  Laurent polynomial rings]{A classification of torsors over  Laurent polynomial rings }

\author{V. Chernousov}
\address{Department of Mathematics, University of Alberta,
    Edmonton, Alberta T6G 2G1, Canada}
\thanks{ V. Chernousov was partially supported by the Canada Research
Chairs Program
and an NSERC research grant} \email{chernous@math.ualberta.ca}

\author{P. Gille}\address{UMR 5208
Institut Camille Jordan - Universit\'e Claude Bernard Lyon 1
43 boulevard du 11 novembre 1918
69622 Villeurbanne cedex - France \newline
   Institute of Mathematics Simion Stoilow of the Romanian Academy,
  Calea Grivitei 21,
 RO-010702 Bucharest, Romania.
}
\thanks{ P. Gille was supported by the Romanian IDEI project PCE$_{-}$2012-4-364 of the Ministry of National Education
CNCS-UEFISCIDI}
\email{gille@math.univ-lyon1.fr}
\author{A. Pianzola}
\address{Department of Mathematics, University of Alberta,
    Edmonton, Alberta T6G 2G1, Canada.
    \newline
 \indent Centro de Altos Estudios en Ciencia Exactas, Avenida de Mayo 866, (1084) Buenos Aires, Argentina.}
\thanks{A. Pianzola wishes to thank NSERC and CONICET for their
continuous support}\email{a.pianzola@gmail.com}

\begin{abstract} Let $R_n$ be the ring of Laurent polynomials in $n$ variables 
over a field $k$ of characteristic zero and let $K_n$ be its fraction field.
Given a linear  algebraic $k$--group  $G$, we show that a
$K_n$-torsor under $G$ which is  unramified with respect to $X={\rm Spec}(R_n)$
extends to a unique toral $R_n$--torsor under $G$. This result, in turn, allows us  to classify all
$G$-torsors over $R_n$.

\smallskip


\end{abstract} 
 
 \maketitle 
 



{\small
\tableofcontents }


\section{Introduction} 

Torsors are the algebraic analogues of the principal homogeneous spaces that one encounters in the 
theory of Lie groups. While the latter are (under natural assumptions) locally trivial (in the usual topology), 
this is not the case for the former: a torsor $E \to X$ under a group scheme $\GG$ over $X$ is not necessarily 
trivial (i.e., isomorphic to 
$\GG$ with $\GG$ acting on itself by right multiplication) when restricted to any non-empty Zariski open subset $U$ of $X.$ 
The reason for this is that the Zariski topology is too coarse. The path out of 
this serious obstacle was initiated by J.-P. Serre (with what is now called the finite \'etale topology) and 
then implemented in full generality (together with an accompanying descent theory) by A. Grothendieck. 
The idea is to have certain morphisms $U \to X$ (e.g. \'etale or flat and  
of finite presentation) replace open immersions as trivializing local data. 

Torsors have played an important role in number theory (Brauer groups, Tate-Shafarevich group, 
Manin obstructions) and in the Langlands program (Ngo's proof of the Fundamental Lemma). 
Somehow surprisingly torsors have been used over the last decade to solve difficult problems in 
infinite dimensional Lie theory (see \cite{GP3} for an extensive list of references. See also \cite{CGP2} and \cite {KLP}). 

The way that torsors arise in this context is the following. The infinite dimensional object $\cL$ under 
consideration (for example, the centreless core of an extended affine Lie algebra [EALA] or a  Lie 
superconformal algebra) has an invariant called the centroid (essentially the linear endomorphisms of 
the object that commute with their multiplication). These centroids are Laurent polynomial rings 
$k[{t_1^{\pm 1}},\cdots, t_n^{\pm1}]$ in finitely many variables over a base field $k$ of characteristic 
$0.$ We will denote this ring by $R_n$, or simply by $R$ if no confusion is possible. 

The object $\cL$ is naturally an $R$-module and it inherits the algebraic structure of $\cL$ (for example $\cL$ 
is a Lie algebra over $R$). It is when $\cL$ is viewed as an object over $R$ that  torsors enter into the picture. 
One can for example classify up to $R$-isomorphism the objects under consideration using non-abelian 
\'etale cohomology. Of course at the end of the day one wants to understand the problem under 
consideration (say a classification) over $k$ and not $R$. Thankfully there is a beautiful theory, known 
as the ``centroid trick", that allows this passage. 

The above discussion motivates why one is interested in the classification of torsors over $R$ under 
a smooth reductive $R$-group scheme. We believe that the understanding of such 
torsors is of its own interest. This is the purpose of the present work. There is an important class of 
torsors under $\GG$ 
called {\it loop torsors} that appear naturally in infinite dimensional Lie theory. (Loop torsors are 
defined over an arbitrary base in \cite{GP3}. They have already caught the attention of researchers 
in other areas. See for example \cite{PZ}.) It is shown in \cite{GP3} that if $E$ is a torsor over $R$ 
under $\GG$, then $E$ being a loop torsor is equivalent to $E$ being {\it toral}, i.e. that the twisted $R$-group 
scheme $^{E}\GG$ admits a maximal torus. This is a remarkable property of the ring $R$: loop and 
toral torsors coincide. 

Toral torsors under reductive group schemes were completely classified in our paper
\cite{CGP2} with the use of Bruhat-Tits theory of buildings (see the Acyclicity Theorem~\ref{acyclicity}). 
We will use toral torsors in what follows (but the reader is asked to keep in mind that these are precisely the 
torsors that arise in infinite dimensional Lie theory). Given a smooth reductive group scheme 
$\GG$ over $R_n$ we want to classify/describe all the isomorphism classes  of $R_n$-torsors 
under $\GG$. Since $\GG$ is smooth reductive they are in natural one-to-one correspondence 
with elements of the pointed set $H^1_{\et}(R_n,\GG).$ We have of course 
by definition a natural inclusion 
$$H^1_{\et, toral}(R_n, \GG)    \, \subseteq \,  H^1_{\et} (R_n, \GG)$$ 
where $H^1_{\et, toral}(R_n, \GG)$ is the subset consisting of (isomorphism) classes of toral $\GG$-torsors.  
From now on by default our topology will be  \'etale; in particular we will denote $H^1_{\et}$ by $H^1$ and 
$H^1_{\et,toral}$ by $H^1_{toral}$. One of our main results 
is the following.
\begin{theorem}\label{main1} Under the above notation  there is a natural bijection
$$
H^1(R_n, \GG) \longleftrightarrow \bigsqcup_{[E] \in H^1_{toral}(R_n, \GG) }  \, H^1_{Zar}(R_n, {^{E}\GG}).$$
\end{theorem}

The content of the above equality could be put in words as follows. Given a torsor $E'$ over $R_n$ under $\GG$ there exists a {\it unique} toral torsor $E$ such that $E'$ is locally isomorphic (in the Zariski topology) to $E.$ 

The proof of this 
result is achieved by a careful analysis (of independent interest) 
of the ramification of the torsors under consideration. We denote by $K_n=k(t_1,\dots,t_n)$  
the fraction field of $R_n$ and set $F_n=k((t_1))\dots ((t_n))$. The precise statement of our other 
main result is the following.

\begin{theorem}\label{main} Let $\GG$ be a smooth affine  $R_n$--group scheme. Assume that either
\smallskip

{\rm (i)}  $\GG$ is reductive and admits a maximal $R_n$-torus (equivalently $\GG$ 
is ``loop reductive'' \cite[cor. 6.3]{GP3});

\smallskip

\noindent
or

\smallskip

{\rm (ii)} there exists a linear (smooth, not necessary connected) algebraic group  $\bG$ over $k$ and a
loop torsor $E$ under $\bG \times_k R_n$
such that $\GG= {^{E}(\bG \times_{k} R_n)}$.

\smallskip

\noindent Then we have natural bijections
$$
H^1_{ toral }(R_n , \GG) \simlgr H^1(K_n , \GG)_{R_n-unr} \simlgr H^1(F_n, \GG)
$$
where $H^1(K_n , \GG)_{R_n-unr}$ stands for the subset of the Galois cohomology set $ H^1(K_n , \GG)$  consisting of  (isomorphism) classes of $\GG$-torsors over $K_n$ extending everywhere 
in codimension one (see \S \ref{sect_def_unr}).
\end{theorem}

We need to explain briefly the assumptions. In both cases (i) and (ii) we consider loop (=toral) group 
schemes because they
play a central role in the classification of $R_n$-torsors (see Theorem~\ref{main1}). 
Also, even though reductive group schemes 
are the main 
interest in this paper,  in case (ii) we include  group schemes over $R_n$ which 
are not necessary ``connected". Far from being a trivial generalization, this case is absolutely essential for applications to infinite dimensional Lie theory. Indeed one is forced to understand twisted forms of $R_n$-Lie algebras $\mathfrak{g} \otimes_k R_n$, where   $\mathfrak{g}$ is a 
split finite dimensional simple Lie algebra over $k$. This leads to torsors under the group $\G \times_k R_n$ where $\G$ is the linear algebraic $k$-group  ${\rm Aut}(\mathfrak{g}).$ Many interesting 
infinite dimensional Lie objects over $k$,  including extended affine Lie algebras 
(a particular case of them are the celebrated affine Kac--Moody Lie algebras) and 
Lie superconformal algebras, follow under the above considerations.


Note that  the special case  $\G=\PGL_d$ (i.e. $R_n$--Azumaya algebras) 
was already quite understood by Brauer group techniques when the base field is algebraically 
closed \cite[\S 4.4]{GP2}. 
 Note also that Theorem~\ref{main} refines  our  acyclicity theorem, i.e. the bijection 
 $H^1_{toral}(R_n, \GG) \simlgr H^1(F_n, \GG)$ 
(\cite[th. 14.1]{CGP2} in case (i) (resp. \cite[th. 8.1]{GP3} in  case (ii)).

\medskip

The structure of the paper is as follows. In  section \ref{section_sorites}, we establish useful generalities
about unramified  functors. Section \ref{section_cohomology}
discusses unramified non-abelian cohomology.
 In section \ref{section_proof}, we prove  Theorem~\ref{main}. Section \ref{section_applications} is devoted to applications including a disjoint union decomposition
for the set  $H^1(R_n , \GG)$ (Theorem~\ref{main1}). Two important particular cases
are considered in detail: the cases of orthogonal groups and
projective linear groups. This illustrates that our main result, which may look rather abstract and remote in appearance, can lead to new very concrete classifications/descriptions of familiar objects.

\section{Unramified functors}\label{section_sorites}

We follow essentially the setting of \cite{CT}.
Let $S$ be a scheme. If $X$ is an integral  $S$-scheme 
we denote by $\kappa(X)$ the   fraction field of $X$.
Let $\calF$ be an $S$-functor, 
that is a contravariant functor $X \mapsto \calF(X)$ 
from the category of $S$--schemes into the category of sets.
If $X$ is  integral normal one defines the following two subsets of $ \calF(\kappa(X)) $:
$$
\calF(\kappa(X))_{X-loc}:=  \bigcap\limits_{ x \in X}
\, \mathrm{Im}\Bigl(  \calF(O_{X,x}) \to  \calF(\kappa(X)) \Bigr)
$$
and
$$
\calF(\kappa(X))_{X-unr}:=  \bigcap\limits_{ x \in X^{(1)}}
\, \mathrm{Im}\Bigl(  \calF(O_{X,x}) \to  \calF(\kappa(X)) \Bigr).
$$
The first subset $\calF(\kappa(X))_{X-loc}$ is called the subset of {\it local classes}
with respect to $X$ and the second one $\calF(\kappa(X))_{X-unr}$ is called the subset of {\it unramified classes }
with respect to $X$.
Obviously, we have the inclusions $$
\calF(\kappa(X))_{X-loc} \, \subseteq \, \calF(\kappa(X))_{X-unr} \, \subseteq \,
\calF(\kappa(X)).
$$

\begin{lemma}\label{lem_funct}
Let  $Y$ is an integral normal scheme over $S.$ Let $f:Y \to X$ be a dominant morphism
of $S$-schemes. Consider the map $\calF(f^*): \calF(\kappa(X)) \to \calF(\kappa(Y))$
 induced by the comorphism $f^*=f^*_{\kappa(X)}:\kappa(X)\to\kappa(Y)$. Then.

\smallskip
  
\noindent (1) 
$\calF(f^*)(\calF(\kappa(X))_{X-loc}) \, \subseteq \, \calF(\kappa(Y))_{Y-loc}$.
 
\smallskip
  
\noindent (2) If $f$ is flat
then 
 $\calF(f^*)(\calF(\kappa(X))_{X-unr}) \, \subseteq \, \calF(\kappa(Y))_{Y-unr}$.
 \end{lemma}

\begin{proof} (1) The comorphism $f^*$ allows us to view $\kappa(X)$ as a subfield 
$\kappa(X) \hookrightarrow \kappa(Y)$ of the field $\kappa(Y)$.
Let $\gamma \in \calF(\kappa(X))_{X-loc}$.
We want to show that its image   $\gamma_{\kappa(Y)}:=\calF(f^*)(\gamma) \in \calF(\kappa(Y))$ 
under the base change is 
local with respect  to $Y$.  
Let $y \in Y$ and put $x=f(y)$. 
The commutative square
$$
\begin{CD}
\kappa(X) & \enskip  \subset \enskip  & \kappa(Y) \\ 
\cup && \cup \\
O_{X,x}& \enskip \subset \enskip  &  O_{Y,y}
\end{CD}
$$
induces a commutative diagram
$$
 \xymatrix{
\calF( \kappa(X))  \ar[r] & \calF( \kappa(Y))  \\ 
\calF( O_{X,x})  \ar[r] \ar[u]& \calF( O_{Y,y} ).\ar[u] 
}
$$
Since $\gamma \in \mathrm{Im}\bigl( \calF( O_{X,x}) \to \calF( \kappa(X)) \bigr)$, it follows
that $\gamma_{\kappa(Y)}$ is contained in $\mathrm{Im}\bigl( \calF( O_{Y,y}) \to \calF( \kappa(Y))\bigr)$.
Thus $\gamma_{\kappa(Y)} \in \calF( \kappa(Y))_{Y-loc}$.
\smallskip

\noindent (2)  Assume now that $\gamma \in \calF( \kappa(X))_{X-unr}$
and let $y \in Y^{(1)}$. As above, we set  $x=f(y)$.
Without loss of generality we may assume that 
$$X=\Spec(A)= \Spec(O_{X,x})$$
and that 
$$Y=\Spec(B)=\Spec(O_{Y,y})$$ where $B$ is a DVR. 
Let $v$ be the discrete valuation  on $\kappa(Y)$ corresponding to  the valuation ring $B$.
For brevity we denote by $K$ (resp. $L$) the fraction field of $A$ (resp. $B$).
If $v(K^\times)=0$, then $K \subset B$ and therefore  
$$\gamma_{\kappa(Y)} \in \mathrm{Im}\bigl( \calF( O_{Y,y}) \to \calF( \kappa(Y)) \bigr).$$

Assume now that $v(K^\times) \not = 0$.
Then ${\mathfrak m}_A B \not = B$ so that 
$A \to B$ is a local morphism. By \cite[0.6.6.2]{EGAI}, since $B$ is flat over $A$ the ring 
$B$ is a faithfully flat $A$--module. 
It follows
 that $A=B \cap K$ (apply \cite[\S\,I.3, \S 5, prop. 10]{BAC} with $F=K$ and $F'=A$. An alternative proof can be given by appealing to \cite[2.1.13]{EGAIV}).
Let $A_v= \{ x \in K^\times \, \mid \, v(x) \geq 0\}$ be
the valuation ring of $v_{\mid K}$. Then $A_v= K \cap B=A$, so that $A$ is a DVR.
This implies that $\gamma \in \mathrm{Im}\bigl( \calF( O_{X,x}) \to \calF( \kappa(X)) \bigr)$, and 
the commutative diagram above yields that
$\gamma_{\kappa(Y)} \in  \mathrm{Im}\bigl( \calF( O_{Y,y}) \to \calF( \kappa(Y)) \bigr)$.
Thus $\gamma_{\kappa(Y)} \in \calF( \kappa(Y))_{Y-unr}$.
\end{proof}

\smallskip

We shall discuss  next the case of  non-abelian cohomology functors, 
but we remark that this technique and considerations can be applied to various interesting functors 
such as Brauer
groups, Witt groups, unramified Galois cohomology...

\section{Non-abelian cohomology}\label{section_cohomology}

\subsection{Some terminology}\label{subsec_terminology}

%
%

Let $X$ be a scheme and let $\GG$ be  an $X$--group scheme.
The pointed set of non-abelian \v Cech cohomology on the flat
 (resp. \'etale,  Zariski) site of $X$ with
 coefficients in $\GG$, is denoted by $H_{fppf}^1(X, \GG)$ (resp.
 $H_{\et}^1(X, \GG)$,  $H_{Zar}^1(X, \GG))$.  These pointed
 sets measure the isomorphism classes of sheaf torsors over $X$ under 
 $\GG$ with respect to the chosen topology (see  \cite[ Ch.\,IV
 \S1]{M} and \cite{DG} for basic definitions and references). If $X = \Spec(R),$ following
 customary usage and depending on the context, we also use the notation
 $H_{fppf}^1(R, \GG)$ instead of $H_{fppf}^1(X, \GG)$. Similarly for
 the \'etale and Zariski sites.

 If  $\GG$ is in addition affine, by faithfully flat descent all of our sheaf torsors are
 representable. They are thus {\it torsors}
 in the usual sense. For a $\GG$-torsor $E$   we denote by  ${^E\GG}$ the twisted form
of $\GG$ by inner automorphisms; it is an affine group scheme over $X$.
 Furthermore, if $\GG$ is smooth
 all torsors are locally trivial for the \'etale topology. In
 particular, $H_{\et}^1(X, \GG) =H_{fppf}^1(X, \GG).$ These
 assumptions on $\GG$ hold in most of the situations that arise
 in our work. Also, as we mentioned in the introduction, by default 
our topology will be \'etale so that instead of $H^1_{\et}(X,\GG)$ we will write
$H^1(X,\GG)$.

 \smallskip

 Given an $X$--group $\GG$ and a morphisms $Y \rightarrow X$ of schemes, 
 we let $\GG_Y$ denote the $Y$--group $\GG \times_X Y$ obtained by 
 base change. For convenience, we will
 denote 
 $H^1(Y,
 \GG_Y)$ by $H^1(Y, \GG)$.

\smallskip

Assuming $\GG$ affine of finite type, a \emph{maximal torus} $\TT$ of $\GG$ is 
a subgroup $X$-scheme $\TT$ of $\GG$ such that $\TT \times_X \ol{\kappa(x)}$ is 
a maximal $\ol{\kappa(x)}$--torus of $\GG \times_X \ol{ \kappa(x)}$ for each
point $x \in X$ \cite[XII.1]{SGA3}.   Here $\ol{\kappa(x)}$ denotes an algebraic closure of 
$\kappa(x)$.  A $\GG$--torsor $E$ is  \emph{toral} if the twisted 
group scheme $^E \GG$ admits a maximal $X$-torus.
We denote by $H_{fppf, toral}^1(X, \GG)$ the subset of 
$H_{fppf}^1(X, \GG)$  consisting of (isomorphism) classes of toral $X$--torsors under $\GG$.

\subsection{Torsion bijection}\label{torsion}
If $E$ is an $X$--torsor under $\GG$ (not necessarily toral),
according to \cite[III.2.6.3.1]{Gi} there exists a natural bijection
$$ 
\tau_E: H_{fppf}^1(X, {^E\GG}) \to H_{fppf}^1(X, \GG),$$
called the torsion bijection, which takes the class of the trivial torsor under $^E\GG$ 
to the class of $E$. It is easy to see that its restriction to classes of toral torsors 
induces a bijection 
$$H_{fppf, toral }^1(X, {^E\GG}) \to H_{fppf,toral}^1(X, \GG).$$


%
%

\subsection{Acyclicity Theorem} The following theorem is the main tool for proving  our main results.
\begin{theorem}\label{acyclicity} Let $\GG$ be a smooth affine  $R_n$--group scheme. Assume that either
\smallskip

{\rm (i)}  $\GG$ is reductive and admits a maximal $R_n$-torus (equivalently $\GG$ 
is ``loop reductive'' \cite[cor. 6.3]{GP3});

\smallskip

\noindent
or

\smallskip

{\rm (ii)} there exists a linear (smooth, not necessary connected) algebraic group  $\bG$ over $k$ and a
loop torsor $E$ under $\bG \times_k R_n$
such that $\GG= {^{E}(\bG \times_{k} R_n)}$.

\smallskip

\noindent Then a natural map
$$
H^1_{ toral }(R_n , \GG) \longrightarrow H^1(F_n, \GG)
$$
is bijective.
\end{theorem}
\begin{proof} See \cite[th. 14.1]{CGP2} in case (i) and \cite[th. 8.1]{GP3} in  case (ii).
\end{proof}

\subsection{Grothendieck--Serre's conjecture}\label{sect_gr}

The following conjecture is due to  
Grothendieck--Serre (\cite[Remarque 3]{Se1}, \cite[Remarque 1.1.a]{Gr}).

\begin{conjecture}\label{conjlocal} Let $R$ be a regular local ring with fraction field
$K$. If  ${\mathfrak G}$  a reductive group scheme over $R$
then the natural map $H^1(R, {\mathfrak G}) \to H^1(K, {\mathfrak G})$ has trivial kernel.
\end{conjecture}

If $R$ contains an infinite field $k$ (of any characteristic), 
the conjecture has been proven by Fedorov--Panin \cite{FP,PSV}. When  
${\mathfrak G} = \G \times_k R$ for some reductive $k$-group $\G$, the so called ``constant" case,   was established before by 
Colliot-Th\'el\`ene and Ojanguren \cite{CTO}. For our considerations we need a similar result
for group schemes which are not necessary ``connected".

\begin{lemma}\label{devissage}
Let $R$ be a regular  local ring with fraction field
$K$. Let ${\mathfrak G}$ be an affine smooth group scheme over $R$
 which is an extension of a finite  twisted constant group scheme $\bF$ over $R$ by a
reductive group scheme ${\mathfrak G}^0$ over $R$.
Assume that Grothendieck--Serre's conjecture holds for ${\mathfrak G}^0$.
Then the map $H^1(R, {\mathfrak G}) \to H^1(K, {\mathfrak G})$ has trivial kernel.
\end{lemma}

\begin{proof} We consider the commutative exact diagram of pointed sets
$$
\begin{CD}
\bF(R) @>>> H^1(R, {\mathfrak G}^0) @>>> H^1(R, {\mathfrak G}) @>>> H^1(R,\bF) \\
@V{\wr}VV @V{\eta}VV @V{\nu}VV @V{\lambda}VV \\
\bF(K) @>>> H^1(K,{\mathfrak G}^0) @>>> H^1(K,{\mathfrak G}) @>>> H^1(K,\bF). \\
\end{CD}
$$
 Since $\lambda$ is injective, an easy diagram chase shows that ${\rm Ker}\,\eta$ 
 surjects onto ${\rm Ker}\,\nu$. But by hypothesis ${\rm Ker}\,\eta$ vanishes, so the assertion follows.
\end{proof}

As a corollary of Fedorov--Panin's theorem we get the following facts.

\begin{corollary}\label{cor_FP}
Let $R$ be a regular local ring
containing an infinite field with fraction field $K$.
Let ${\mathfrak G}$ be an affine smooth  group scheme over $R$
 which is an extension of a finite  twisted constant group scheme  over $R$ by a
reductive group scheme 
over $R$.
Then the natural map $H^1(R, {\mathfrak G}) \to H^1(K, {\mathfrak G})$ has trivial kernel.
\end{corollary}

\begin{corollary}\label{zar}
Let $X$ be an integral smooth affine variety over an infinite field with  function field $K$. 
Let ${\mathfrak G}$ be an affine smooth  group scheme over $X$
 which is an extension of a finite  twisted constant group scheme  over $X$ by a
reductive group scheme 
over $X$. Then the sequence of pointed sets 
$$
1\longrightarrow H^1_{Zar}(X,\GG)\longrightarrow H^1(X,\GG)\longrightarrow H^1(K,\GG)
$$
is exact.
\end{corollary}

\subsection{Rational torsors everywhere locally defined}\label{sect_def_unr}

For  a smooth affine group scheme $\GG$ over an integral normal scheme $X$ 
Colliot-Th\'el\`ene and Sansuc  \cite[\S 6]{CTS}
introduced the following sets:
$$
D^{\mathfrak G}(X):={\rm Im}\Bigl( H^1(X,{\mathfrak G}) \to H^1( \kappa(X), {\mathfrak G}) \Bigr),
$$
$$
H^1(\kappa(X),{\mathfrak G})_{X-loc} := \bigcap\limits_{ x \in X}
\, D^{\mathfrak G}( {\cal O}_{X,x})
\enskip \subseteq \enskip H^1( \kappa(X), {\mathfrak G} ),
$$
and
$$
H^1(\kappa(X),{\mathfrak G})_{X-unr} := \bigcap\limits_{ x \in X^{(1)}}
\, D^{\mathfrak G}( {\cal O}_{X,x})
\enskip \subseteq \enskip H^1( \kappa(X), {\mathfrak G} ).
$$
Clearly, we have the inclusions 
$$D^{\mathfrak G}(X) \subseteq H^1(\kappa(X),{\mathfrak G})_{X-loc}\subseteq
H^1(\kappa(X),{\mathfrak G})_{X-unr}.$$
In our terminology introduced in \S\,\ref{section_sorites} the
two last sets are nothing but the local and unramified subsets with respect to $X$ 
attached to the functor $\calF$ given by $\calF(Y)= H^1(Y, {\mathfrak G})$ for each $X$-scheme $Y$.
Unramified classes have the following geometrical characterization.

\begin{lemma}\label{lem_unr} 
  Let $\gamma \in H^1(\kappa(X),{\mathfrak G})$. Then $\gamma \in 
H^1(\kappa(X),{\mathfrak G})_{X-unr}$ if and only if 
there exists an open subset $U$ of $X$ and a class $\widetilde \gamma \in H^1_{\et}(U, {\mathfrak G})$ 
such that 

\smallskip

{\rm (i)} $\gamma=(\widetilde \gamma)_{\kappa(X)}$; 

\smallskip

{\rm (ii)} $X^{(1)} \subset U$.

\end{lemma}

\begin{proof} In one direction the statement is obvious.The other one was treated in \cite[cor. A.8]{GP2}.
\end{proof}

A special case of Lemma \ref{lem_funct} is then 
the following.

\begin{lemma}\label{lem_funct2}
Let $f:Y \to X$ be a dominant morphism of integral and normal $S$-schemes. Let 
$$\calF(f^*):  H^1( \kappa(X), {\mathfrak G} ) \to H^1( \kappa(Y), {\mathfrak G}_Y )$$ 
be the  map induced by the comorphism $f^*:\kappa(X)\to\kappa(Y)$. Then.

\smallskip
  
\noindent 
{\rm (1)}  $\calF(f^*)(H^1(\kappa(X),\GG)_{X-loc}) \, \subseteq \,  H^1(\kappa(Y),\GG)_{Y-loc}.$
 
\smallskip
  
\noindent 
{\rm (2)} If  $f$ is flat
then 
 $$\calF(f^*)(H^1(\kappa(X),\GG)_{X-unr}) \, \subseteq \, H^1(\kappa(Y),\GG)_{Y-unr}.$$

\end{lemma}


\begin{remark}
{\rm If $X$ is regular and the ``purity conjecture'' holds for  $\GG$ and local rings of $X$,  
i.e. $D_{\mathfrak G}(O_{X,x})= H^1(\kappa(X), {\mathfrak G})_{O_{X,x}-unr}$ for all points $x\in X$, 
then assertion (1) implies that  (2) holds without  flatness assumption for $f$.
}
\end{remark}


We now combine  earlier work by Colliot-Th\'el\`ene/Sansuc
and Nisnevich theorem~\cite{N}   on Grothendieck-Serre conjecture
for reductive groups over DVR.

\begin{proposition}\label{prop_special} 
Assume that $X$ is a regular integral scheme and that 
${\mathfrak G}$ is an extension of a finite twisted constant 
group scheme by a reductive group scheme ${\mathfrak G}^0$. 
Then

\smallskip

\noindent {\rm (1)} 
$D^{\mathfrak G}$  defines a contravariant
functor for the category of regular integral $X$-schemes.

\smallskip

\noindent 
{\rm (2)}  If $X=\Spec(k)$, then $H^1(\, \, ,{\mathfrak G})_{loc}$ 
defines a contravariant
functor for the category of smooth integral $k$--varieties.

\end{proposition}

\begin{proof} 
Nisnevich's theorem states that  if $A$ is a DVR and ${\bf G}$ is a reductive group
over $A$ then the natural map $H^1(R,{\bf G})\to H^1(K,{\bf G})$, where $K$ is a fraction field
of $A$, is injective. 
By Lemma \ref{devissage}, it holds more generally  for a group $\bG$ which is an extension of 
a finite twisted  constant group
by a reductive group.
In particular, if $\gamma_1,\gamma_2 \in H^1(A,{\bf G})$ have the same 
image in $H^1(K,{\bf G})$, 
they have the same specialization modulo the maximal ideal of 
$A$. By \cite[6.6.1]{CTS}, this specialization property holds more generally over
an arbitrary  valuation ring. Then the assertions follow from \cite[6.6.3]{CTS} 
and \cite[proposition 2.1.10]{CT}.
%
%
%
%
\end{proof}


\section{Proof of Theorem~\ref{main}}\label{section_proof}

Let $\GG$ be a smooth affine group scheme over $R_n$ satisfying condition (i)
or (ii)
in Theorem \ref{main}. 
Clearly,
$$
{\rm Im}\,[H^1( R_n,\GG) \to H^1(K_n,\GG)] \, \subset \,  H^1(K_n,\GG)_{R_n-unr},
$$
so that we have
the factorization
$$
\leqno{\bf (*)} \qquad \quad 
H^1_{ toral}( R_n,\GG) \stackrel{\phi}{\longrightarrow} H^1(K_n,\GG)_{R_n-unr} 
\stackrel{\psi}{\longrightarrow} H^1(F_n,\GG).
$$
The Acyclicity Theorem~\ref{acyclicity} states that the composite map $\psi \circ \phi$
is bijective;
in particular, $\phi$ is injective and $\psi$ is surjective.


\begin{lemma}\label{equivalence} The following are equivalent.

\smallskip

\noindent
{\rm (i)} $\phi$ is bijective.

\smallskip

\noindent
{\rm (ii)} The  map 
$$
^E\psi:H^1(K_n, {^E\GG})_{R_n-unr} \longrightarrow H^1(F_n,{^E\GG})
$$ has trivial kernel 
for all toral $R_n$--torsors $E$ under $\GG$.
\end{lemma}

\begin{proof}
$(i) \Longrightarrow (ii)$: Assume  that $\phi$ is bijective. The above
factorization $(*)$ and the bijectivity of $\psi\circ\phi$
yield that we have bijections  
$$H^1_{ toral}( R_n,\GG)  \simlgr H^1(K_n,\GG)_{R_n-unr} \simlgr H^1(F_n,\HH).$$
Let now $E$ be a toral $R_n$--torsor under $\GG$. Recall  that the torsion bijection map 
$$
\tau_E: H^1( R_n,{^E\GG}) \simlgr H^1( R_n,\GG)$$ induces a bijection 
$$
H^1_{ toral}( R_n,{^E\GG}) \simlgr H^1_{ toral}( R_n,\GG)$$ (see \S\,\ref{torsion}).
The commutative diagram of torsion bijections
$$
 \xymatrix{
H^1_{ toral}( R_n,\GG) \ar[r]^{\sim} &  H^1(K_n,\GG)_{R_n-unr} \ar[r]^{\sim} & H^1(F_n,\GG) \\ 
H^1_{ toral}( R_n,{^E\GG}) \ar[r] \ar[u]^{\wr} &  H^1(K_n,{^E\GG})_{R_n-unr} \ar[r]^{^E\psi} \ar[u]^{\wr} & 
H^1(F_n,{^E\GG}) \ar[u]^{\wr} \\ 
}
$$
shows that $^E\psi$ is bijective, so a fortiori has trivial kernel.

\smallskip

\noindent $(ii) \Longrightarrow (i)$: We have noticed above  that $\phi$ is injective, 
hence it remains to prove its surjectivity only. Let $[E']\in H^1(K_n,\GG)_{R_n-unr}$. Since $\psi\circ\phi$
is bijective there exists a class $[E]\in H^1(R_n,\GG)$ such that $\psi([E'])=\psi(\phi([E]))$.
It follows from the above commutative diagram that under the torsion bijection 
$H^1(K_n,{^E\GG})_{R_n-unr}\to H^1(K,\GG)_{R_n-unr}$ the class $[E']$ corresponds to an element
in $H^1(K_n,{^E\GG})_{R_n-unr}$ lying in the kernel of $^E\psi$. Since by our hypothesis ${\rm Ker}({^E\psi})=1$,
this implies that the class $[E']$ corresponds to the trivial one in $H^1(K_n,{^E\GG})_{R_n-unr}$
or equivalently $\phi([E])=[E']$.
%
\end{proof}


Note that  hypothesis (i) and (ii) in Theorem~\ref{main} 
are stable with respect to twisting by a toral $R_n$-torsor under  $\GG$. Therefore
the above lemma reduces the proof of Theorem~\ref{main}
to showing that for all group schemes $\HH$ over $R_n$ satisfying conditions (i) or (ii) in 
Theorem~\ref{main} a natural map
$$\psi: H^1(K_n, \HH)_{R_n-unr} \longrightarrow H^1(F_n,\HH)$$ 
has trivial kernel.  
To prove  this fact we  proceed  by induction on $n\geq 1$ by allowing the base field $k$ to vary.

\smallskip

\noindent{$n=1$:}   Since we are in dimension one, by Lemma \ref{lem_unr}  
the map   
$$H^1( R_1, \HH) \longrightarrow H^1( K_1, \HH)_{R_1-unr}$$ is onto. Therefore, 
 Lemma~\ref{equivalence} applied to the group scheme $\GG=\HH$ and the trivial torsor $E=1$ shows that  
$\psi$ has trivial kernel.


\smallskip

\noindent{$n\geq 2$:} Consider the following field tower: 
$$K_n\subset F_{n-1}(t_n)\subset F_n.$$
Let $\gamma\in {\rm Ker}(\psi)$ and let
$\gamma'$ be its image in $H^1( F_{n-1}(t_n), \HH)$. 
Since  the morphism of affine schemes $\Spec( F_{n-1}[t_n^{\pm 1}]) \to \Spec(R_n)$ is flat and dominant,
by Lemma \ref{lem_funct2}\,(2) 
we have
$$
\gamma' \in H^1\bigl( F_{n-1}(t_n), \HH\bigr)_{F_{n-1}[t_n^{\pm 1}]-unr}  .
$$
Since $F_n=F_{n-1}((t_n))$ and $\gamma_{F_n}=1$, 
we then conclude that
$$
\gamma' \in  {\rm Ker}\Bigl( H^1\bigl( F_{n-1}(t_n), \HH \bigr)_{F_{n-1}[t_n^{\pm 1}]-unr}  \to H^1\bigl(F_{n-1}((t_n)),\HH \bigr) \Bigr).
$$
But according to the case $n=1$ (applied to the base field $F_{n-1}$) the last kernel is trivial. 
Thus  $\gamma'=1$, i.e. 
$$
\leqno{(**)} \qquad \qquad 
\gamma \in {\rm Ker}\Bigl( H^1\bigl( K_n, \HH)_{R_n-unr} \to H^1( F_{n-1}(t_n), \HH \bigr) \Bigr) .
$$
Now, we observe that the field  $F_{n-1}(t_n)= k((t_1))\dots ((t_{n-1}))(t_n)$
embeds into $k(t_n)((t_1))\dots((t_{n-1}))=k'((t_1))\dots((t_{n-1}))$ with $k'=k(t_n)$, 
so that we have a commutative diagram
$$
\begin{CD}
H^1\bigl( K_n, \HH \bigr)_{R_n-unr} @>>> H^1\bigl( F_{n-1}(t_n), \HH \bigr) \\
\cap  && @VVV \\
 H^1\bigl( k'(t_1,...,t_{n-1}),  \HH \bigr)_{R_{n-1} \otimes_k k'-unr} @>>> H^1\bigl( k'((t_1))\dots ((t_{n-1})) , \HH \bigr) ,
\end{CD}
$$
where the left vertical inclusion again is  due to Lemma \ref{lem_funct2}\,(2).
By the induction hypothesis applied to the base field $k'$,  
 the bottom horizontal map has trivial kernel.
Therefore the top horizontal map has trivial kernel as well.
By  $(**)$, this implies $\gamma=1$.

\section{Applications}\label{section_applications}

\subsection{A disjoint union decomposition for $R_n$--torsors}

We shall use now our Theorem~\ref{main} and Fedorov-Panin's theorem to prove Theorem~\ref{main1}. In fact we 
will prove a little bit more general result by allowing $\GG$ to be  any group scheme
 satisfying conditions of Theorem~\ref{main}.

\begin{theorem}\label{cor_main} Let $\GG$ be as in Theorem \ref{main}. Then there is a natural bijection
$$
\bigsqcup_{[E] \in H^1_{ toral}(R_n, \GG) }   H^1_{Zar} (R_n, {^{E}\GG}) \enskip  \buildrel{\Theta} \over \simlgr
\enskip H^1_{}(R_n, \GG).
$$
\end{theorem}
\begin{proof}
Recall first that the torsion bijection $\tau_E$ (see \S\,\ref{torsion}) allows us to embed 
$$
H^1_{Zar} (R_n, {^{E}\GG})
\hookrightarrow H^1_{}(R_n, {^{E}\GG})
\buildrel \tau_{E} \over \simlgr  H^1(R_n, \GG).
$$
and this, in turn, induces a natural map
$$
\Theta: \bigsqcup_{[E] \in H^1_{toral}(R_n, \GG) }   H^1_{Zar} (R_n, {^{E}\GG}) \enskip  \to 
\enskip H^1_{\et}(R_n, \GG).
$$
\noindent{\it Surjectivity of $\Theta$.}
Assume $[\gamma] \in H^1(R_n, \GG)$. Since the generic class $\gamma_{K_n} \in H^1(K_n, \GG)$ is 
$R_n$--unramified by
Theorem~\ref{main}  there is a unique toral class $[E] \in H^1(R_n , \GG)_{toral}$
such that $[E]_{K_n}=\gamma_{K_n}$. Consider the following commutative diagram 
$$
\begin{CD}
& &H^1(R_n, \GG) @>>> H^1(K_n, \GG) \\
& &@A{\tau_{E}}A{\wr}A   @A{\tau_{E} }A{\wr}A \\
H^1_{Zar}(R_n,{^E\GG}) 
& \,\,\stackrel{\iota}{\hookrightarrow}\,\, & H^1(R_n, {^{E}\GG}) @>>> H^1(K_n, {^{E}\GG})
\end{CD}
$$
with an exact  bottom horizontal line (see Corollary~\ref{zar}).
It follows from the diagram that that 
$\tau_{E}^{-1}(\gamma) \in {\rm Im}\,\iota$.
Hence there exists a unique class $\eta \in H^1_{Zar}(R_n, {^{E}\GG}) $ such that 
$\eta=\tau_{E}^{-1}(\gamma)$. By construction, $\Theta(\eta)=\gamma$.

\smallskip

\noindent{\it Injectivity of $\Theta$.} Let $E,E'$ be two toral torsors under $\GG$ and
let $$
\eta \in H^1_{Zar}(R_n, {^E\GG}),\ \  \eta' \in H^1_{Zar}(R_n, {^{E'}\GG})$$
be such that $\Theta(\eta)=\Theta(\eta')$, i.e. $\tau_{E}( \eta)= \tau_{E'}( \eta') \in H^1(R_n, \GG)$. 
Since $\eta$ is locally trivial in the Zariski topology we conclude that
$$\tau_{E}( \eta)_{K_n}= \tau_{E}( 1)_{K_n}=[E]_{K_n}$$ and similarly for $E'$. It follows that 
 $$
 [E]_{K_n}=[E']_{K_n} \in H^1(K_n, \GG)_{R_n-unr}$$
 and therefore, by Theorem~\ref{main},  we have $[E]=[E']$. But $\tau_E=\tau_{E'}$ is bijective.
 Therefore, $\eta=\eta'$.
\end{proof}

\smallskip

\begin{example}{\rm 
 Let $\bG$ be a reductive $k$--group with the property that all of its semisimple quotients are isotropic
 (for example, $k$-split). 
By a result of Raghunathan \cite[th. B]{R} one has
$$
H^1_{Zar}( {\bf A}_k^n, \bG)=1$$ for all $n \geq 0$.
According to \cite[prop. 2.2]{GP2},  any Zariski locally trivial $\bG$-torsor over $R_n$  
can be extended 
to a $\bG$-torsor over ${\bf A}_k^n$.
Therefore  $H^1_{Zar}(R_n, \bG)=1$. 
This implies that the map $H^1(R_n,\bG) \to H^1(K_n,\bG)$ has trivial kernel.
In other words,  rationally trivial $R_n$--torsors under $\bG$ are trivial.
}
\end{example}

We expect that the similar result holds in a more general case.

\begin{conjecture}
Let $\GG$ be a loop reductive group scheme over $R_n$. 
If  all semisimple quotients of $\GG$ are isotropic,
then $H^1_{Zar}(R_n,\GG)=1$.
\end{conjecture}

\begin{remark} 
Note that Artamonov's freeness result \cite{A} as well as Parimala's result \cite{P} for quadratic forms over $R_2$ are particular special cases of this 
conjecture

\end{remark}

 Theorem~\ref{cor_main} gives a classification of all $\GG$-torsors that involves first
the description of all its toral torsors  and then to studying  locally trivial in Zariski topology  
torsors under its twisted toral forms. In the following two subsections we show how our theorem
works in particular
cases of orthogonal groups and projective linear groups.

\subsection{The case of  orthogonal groups}

Let $q_{split}$ be a split quadratic form over $k$ 
of dimension  $d \geq 1$ and let
${\bf O}(q_{split})$ be the corresponding split orthogonal group.
It is well-known that $H^1(R_n, {\bf O}(q_{split}))$ classifies non-singular  quadratic 
$R_n$--forms of dimension $d$ \cite[III.5.2]{DG}. This allows us to identify classes of torsors
under $\GG={\bf O}(q_{split,R_n})$ with classes of $d$-dimensional quadratic forms over 
the ring $R_n$.

\begin{proposition} For each subset $I \subseteq \{1,...,d\}$, we put 
$t_I= \prod_{i \in I} t_i\in R_n^{\times}$ with the convention $t_\emptyset = 1$.

\smallskip

\noindent {\rm (1)} Each class in $H^1_{toral}(R_n, {\bf O}(q_{split}))$ contains a unique $R_n$--quadratic form
of the shape
$$
  \bigoplus\limits_{I \subseteq \{1,...,d\}} \,  \langle  t_I \rangle \otimes q_{I, R_n}
$$
where all $q_I$'s with $I\not=\emptyset$ are non-singular anisotropic quadratic forms over $k$ such that 
%
%
%
%
%
$$d =\bigoplus\limits_{I \subseteq \{1,...,d\}} \, \dim(q_I).$$ 

%

\smallskip 

\noindent {\rm (2)} Let $q$ be a non--singular quadratic $R_n$--form of dimension $d$.
Then there exists a unique quadratic $R_n$-form $q_{loop}$ as in $(1)$
such that $q$ is a Zariski $R_n$--form of  $q_{loop}$.  
Furthermore, $q$ is isometric  to $q_{loop}$ if and only if
$q$ is diagonalizable.
\end{proposition}


\begin{proof}
 (1) By Acyclicity Theorem~\ref{acyclicity} it suffices to compute $H^1(F_n,{\bf O}(q_{split}))$ 
 or equivalently isometry classes of $d$-dimensional quadratic forms over $F_n$.

Let $q$ be such quadratic form.  We want to show that it is as in (1). 
By the Witt theorem we may assume without loss of generality that $q$ is anisotropic.
 We proceed by induction on $n\geq 0$. The case $n=0$ is obvious. 
 Assume that  $n\geq 1$. Note that $F_n=F_{n-1}((t_n))$. Springer's decomposition \cite[\S 19]{EKM}, 
 then shows that $q \cong  q' \oplus \langle t_n \rangle \,  q''$ 
 where $q'$ and $q''$ are (unique) anisotropic quadratic forms over $F_{n-1}$.
 By induction on $n$, $q'$ and $q''$ are of the required form, hence we the assertion for $q$ follows. 
 The unicity is clear.

\smallskip

\noindent (2) The first assertion follows from Theorem~\ref{cor_main}.
If $q$ is isometric  to $q_{loop}$, then $q$ is diagonalizable since so is $q_{loop}$.
Conversely, assume that $q$ is diagonalizable: $q=\langle b_1,\ldots,b_d\,\rangle$.  Since 
$$
R_n^\times/ (R_n^\times)^2 \simlgr k^\times/ (k^\times)^2 \times 
\langle t_1^{\epsilon_1}\ldots t_n^{\epsilon_n}\ |\ \epsilon_1,\ldots,\epsilon_n=0,1\,\rangle.
$$
 all coefficients of $q$ are of the shape $b_i=a_i \, t_{I_i} $ with 
 $I_i\subseteq \{1,\ldots,d\}$ and $a_i\in k^{\times}$.
Then we can renumber $b_1,\ldots,b_d$   in such a way that $q$ is as in (1) and we are done.
\end{proof}
\begin{corollary} Toral classes in $H^1(R_n,{\bf O}(q_{split}))$ correspond to diagonalizable
$R_n$-quadratic forms.
\end{corollary}

\subsection{The projective linear case}
Let $\GG= \PGL_d$.
The set $H^1(R_n, \PGL_d)$ classifies 
Azumaya $R_n$-algebras of degree $d$ \cite[prop. 2.5.3.13]{CF}. 
Recall also that if $\calA$ is 
an Azumaya algebra over $R_n$ of degree $d,$ 
then there is a one-to-one correspondence 
between commutative \'etale  $R_n$--subalgebras of $\calA$ of dimension $d$ and
maximal $R_n$--tori of the group scheme $\PGL_1(\calA).$ (This is a 
general well-known fact \cite[XIV.3.21.(b)]{SGA3}. Indeed, if $S$ is a commutative  \'etale $R_n$--subalgebra of $\calA$
of degree $d$ then $\TT=R_{S/R_n}({\bf G}_m)/{\bf G}_m$ is a maximal $R_n$--torus
of $\PGL_1(\calA)$. Conversely, to 
a maximal $R_n$-subtorus $\TT$ of $\PGL_1(\calA)$ one associates 
the $R_n$--subalgebra  $\calA^\TT$ of fixed points under the natural action of
$\TT$. This subalgebra $\calA^{\TT}$ has the required properties locally and hence globally).

From the above discussion it follows that $H^1_{toral}(R_n,\PGL_d)$ consists of isomorphism classes
of Azumaya $R_n$--algebras $\calA$ of degree $d$ having 
\'etale commutative $R_n$--subalgebras of dimension  $d$.


\smallskip

We pass to description of locally trivial torsors under $\PGL_1(\calA)$. 
\begin{lemma} Let $\calA$ be an Azumaya algebra over $R_n$.
 Then the natural map  $H^1_{Zar}(R_n, \GL_1(\calA)) \to  H^1_{Zar} (R_n, \PGL_1(\calA))$
 is bijective.
\end{lemma}

\begin{proof} The exact sequence 
$$
1 \to {\rm G}_m \to \GL_1(\calA) \to \PGL_1(\calA) \to 1$$
gives rise to a commutative diagram with exact rows (of pointed sets)
$$
 \begin{CD}
1@>>>  H^1(R_n, \GL_1(\calA)) @>{\phi}>>  H^1(R_n, \PGL_1(\calA))@>> > \Br(R_n)  \\ 
 && @VVV @V{\psi_1}VV @V{\psi_2}VV\\
& & H^1(K_n,\GL_1(A))=1 @>>> H^1(K_n, \PGL_1(\calA))@>>> \Br(K_n) . 
\end{CD}
$$
Every class in $H^1(R_n,\GL_1(\calA)$ is rationally trivial, hence  by Fedorov--Panin's result \cite{FP,PSV} 
it is locally trivial in the Zariski topology. In other words
$$H^1(R_n,\GL_1(\calA))=H^1_{Zar}(R_n,\GL_1(\calA)).$$
Clearly 
$$
\phi(H^1_{Zar}(R_n,\GL_1(\calA))) \, \subseteq \,  H^1_{Zar}(R_n,\PGL_1(\calA)).
$$
Conversely, let $\gamma\in H^1_{Zar}(R_n,\PGL_1(\calA))$. Since $\psi_1(\gamma)=1$ and
since $\psi_2$ is injective, we obtain $\gamma\in {\rm Im}(\phi)$.

Finally, it remains to note that
the above diagram shows that $\phi$ has trivial kernel and this is true for all Azumaya algebras over $R_n$. 
Then the standard twisting argument enables us to conclude
that 
 $\phi$ is injective.
\end{proof}


Thus, the disjoint union decomposition of the set of isomorphism classes of torsors under $\PGL_d$ becomes  
$$
\bigsqcup_{[\calA] \in H^1_{ toral}(R_n, \PGL_d) }   
H^1_{Zar} (R_n, \GL_1(\calA)) \enskip  \buildrel{\Theta} \over \simlgr
\enskip H^1(R_n, \PGL_d).
$$
In general we can't say much about the subset $H^1_{Zar} (R_n, \GL_1(\calA))$.  
Recall only that it classifies right invertible $\calA$-modules
(so that the above decomposition is coherent with \cite[prop. 4.8]{GP2}).
More precisely, if $\calA$ is a toral Azumaya $R_n$--algebra and $\calL$ is 
a right invertible $\calA$-module, then the class $[\calL]$
corresponds to the class of the Azumaya algebra $\mathrm{End}_{\calA}(\calL)$ under the map $\Theta$.


However, when the base field $k$ is algebraically closed, toral Azumaya algebras 
over $R_n$ are easy to classify explicitly, and also  
much more information about Zariski trivial torsors is available 
due to Artamonov's freeness statements \cite{A}.

More precisely,  let $k$ be an algebraically closed field. 
Choose a coherent system of primitive roots of unity $(\zeta_n)_{n \geq 1}$ in $k$.
Given integers $r,s$ satisfying $1 \leq r \leq s$, we let  
$A(x,y)_r^s$ denote the Azumaya algebra of degree $s$ over the Laurent polynomial ring 
$k[x^{\pm 1}, y^{\pm 1}]$ defined by a presentation
$$
X^s=x, \, Y^s= y, \, YX= \zeta_s^r \,  XY.
$$

\begin{lemma}\label{division} Let $s_1,\ldots,s_m,r_1,\ldots r_m$  be positive integers such that
$(s_i,r_i)=1$ for all $i=1,\ldots,m$. Then the Azumaya algebra 
$$\calA=A(t_1,t_2)^{s_1}_{r_1} \otimes 
A(t_3,t_4)^{s_2}_{r_2} \otimes
\cdots A(t_{2m-1},t_{2m})^{s_m}_{r_m}$$ is a division algebra.
\end{lemma} 
\begin{proof} Indeed, using the residue method it is easy to see that
it is a division algebra even over the field $F_{2m}=k((t_1))\ldots((t_{2m}))$.
\end{proof}

We next recall that the group $\GL_d(\Z)$ acts in a natural way on the ring $R_n$,
hence it acts on 
$H^1(R_n, \PGL_d)$ (for details see \cite[8.4]{GP3}).

\begin{theorem} Assume that $k$ is algebraically closed.
 
 \smallskip

\noindent 
{\rm (1)}  $H^1_{toral}(R_n, \PGL_d)$ consists of  $\GL_d(\Z)$--orbits of classes of $R_n$--algebras
of the following shape:
$$
\calA= M_{s_0}(R_n) \otimes_{R_n} A(t_1,t_2)^{s_1}_{r_1} \otimes A(t_3,t_4)^{s_2}_{r_2} \otimes
\cdots A(t_{2m-1},t_{2m})^{s_m}_{r_m}
$$
where the integers $m$, $s_0, s_1,r_1,\dots,s_m, r_m$ satisfy the following conditions:

\smallskip

{\rm (i)} $0 \leq 2m \leq n$; $1 \leq r_i \leq s_i$  and $(r_i,s_i)=1$ for all $i=1,...,m$;

\smallskip

{\rm (ii)} $s_0 \geq 1$ and $s_0 s_1 \dots  s_m=d$.

\smallskip

\noindent 
{\rm (2)} $\calA \otimes_{R_n} F_n$ is division if and only if $s_0=1$.

\smallskip

\noindent 
{\rm (3)} Let $\calA$ be an Azumaya algebra as in (1).
If $s_0\geq 2$ we have $$H^1_{Zar} (R_n, \GL_1(\calA))=1.$$

\end{theorem}

\begin{proof}(1) If $\calA \otimes_{R_n} F_n$ is division, this is \cite[th. 4.7]{GP3}.  
   Assume now that  $\calA \otimes_{R_n} F_n$ is not division.
 By Wedderburn's theorem there exists an integer $s \geq 1$ and 
 a central division $F_n$--algebra $A'$ such that 
 $\calA \otimes_{R_n} F_n\cong M_s(A')$.
 The acyclicity theorem provides a toral $R_n$--Azumaya algebra $\calA'$ such that 
 $\calA' \otimes_{R_n} F_n\cong A'$. Since $M_s(\calA')$ and $\calA$ are isomorphic
 over $F_n$, the acyclicity theorem again shows that $\calA \cong M_s(\calA')$.
 It remains to note that $\calA'$ being a division algebra over $R_n$ is of the required form
 by the first case.

\smallskip

\noindent
(2) The assertion follows from Lemma~\ref{division}.

\smallskip

\noindent (3) Write  $\calA= M_s( \calA')$ with $\calA'$ is division and  $s \geq 2$.
Morita equivalence provides a one-to-one correspondence between
invertible $\calA$--modules and finitely generated projective $\calA'$--modules
of relative rank $s$. Artamonov's result~\cite{A} states  that those $\calA'$--modules are free \cite[cor. 3]{A}
since $s\geq 2$, so that invertible $\calA$--modules  are free as well. This implies  $H^1_{Zar} (R_n, \GL_1(\calA))=1$.
\end{proof}

\begin{remarks}{\rm

\noindent (a) The third statement shows that for each invertible $\calA$--module $P$ 
the module $P \oplus \calA$ is free.

\smallskip

\noindent (b) The third statement refines Steinmetz's results in the $2$-dimension case ($n=2$) 
\cite[th. 4.8]{St} where  the case  $s\geq 3$ was considered only. Note that 
\cite{St}  provides some other cases for classical groups when 
all Zariski locally trivial torsors are trivial.
}
\end{remarks}

\subsection{Applications to $R_n$--Lie algebras}

We next consider  the special case  $\bG=\Aut(\gg)$
where $\gg$ is a split simple Lie algebra over $k$ of finite dimension.
For such group the set  $H^1_{\et}(R_n, \bG)$ classifies  $R_n$-forms of the Lie algebra $\gg \otimes_k R_n$ and
$H^1_{ toral }(R_n, \bG)$ classifies loop objects, i.e. those which arise from
loop cocycles. More precisely, by \cite[\S 6]{GP3}
we have   
$$
\mathrm{Im}\Bigl( H^1\bigl( \pi_1(R_n,1), \bG(k_s)) \to H^1(R_n, \bG) \Bigr)
= H^1_{ toral }(R_n, \bG).$$   Theorems~\ref{main} and  \ref{cor_main}
have the following consequences.

\begin{corollary} {\rm (1)} Let $\widetilde{\cL}$ be a $K_n$--form of the Lie algebra $\gg \otimes_k K_n$.
If $\widetilde{\cL}$ is unramified, i.e. extends everywhere in codimension one, 
then $\widetilde{\cL}$ is isomorphic to the generic fiber of a unique multiloop Lie algebra $\cL.$ (Of course $\cL$, being a multiloop algebra, is a twisted from of the $R_n$--Lie algebra  
$\gg  \otimes_k R_n$).

\smallskip

\noindent 
{\rm (2)} Let $\calL$ be any $R_n$--form of $\gg \otimes_R R_n$. Then there exists a (unique  up to $R_n$--isomorphism) 
multiloop Lie algebra
$\calL_{loop}$ over $R_n$ such that $\calL$ is a Zariski  $R_n$--form of $\calL_{loop}$.

\end{corollary}

\medskip


\bigskip




\begin{thebibliography}{99}









\bibitem[A]{A} V. A. Artamonov, {\it  Projective modules over crossed products},  J. Algebra 
{\bf 173} (1995),  696--714. 








\bibitem[Bbk]{BAC} N. Bourbaki, {\it Alg\`ebre commutative}, 
Ch. 1 \`a 4, Springer.









\bibitem[CF]{CF} B. Calm\`es, J. Fasel, {\it Groupes classiques}, to appear
    in Autour des sch\'emas en groupes, vol II, Panoramas et Synth\`eses (2015).


\bibitem[CGP1]{CGP1} V. Chernousov, P. Gille and A. Pianzola, {\it
Torsors over  the punctured affine line},  
American Journal of Mathematics  {\bf 134} (2012), 1541--1583.


\bibitem[CGP2]{CGP2} V. Chernousov, P. Gille and A. Pianzola, {\it
Conjugacy theorems for loop reductive group schemes and Lie algebras},
Bulletin of Mathematical Sciences  {\bf 4} (2014), 281--324. 





\bibitem[CT]{CT} J.--L. Colliot--Th\'el\`ene,  {\it Birational invariants, purity and
the Gersten conjecture},  $K$-theory and algebraic geometry: connections with quadratic 
forms and division algebras (Santa Barbara, CA, 1992),  1--64, Proc. Sympos. Pure Math. {\bf 58.1} (1995), AMS.


\bibitem[CTO]{CTO} J.--L. Colliot--Th\'el\`ene, M. Ojanguren, {\it Espaces
principaux homog\`enes localement triviaux}, I.H.\'E.S. Publ. Math. {\bf 75} (1992), 97--122.


\bibitem[CTS]{CTS} J.--L. Colliot--Th\'el\`ene, J.--J. Sansuc, {\it Fibr\'es
quadratiques et composantes connexes r\'eelles}, Math. Annalen {\bf
244} (1979), 105--134.



\bibitem[DG]{DG} M. Demazure et P. Gabriel,
{\it Groupes alg\'ebriques}, Masson (1970).


\bibitem[EKM]{EKM} R. Elman, N. Karpenko, A. Merkurjev, 
{\it The algebraic and geometric theory of quadratic forms}, 
American Mathematical Society Colloquium Publications, 56. American Mathematical Society, Providence, RI, 2008. 


\bibitem[EGAI]{EGAI} A. Grothendieck (avec la collaboration de J.
Dieudonn\'e), {\it El\'ements de G\'eom\'etrie Alg\'ebrique I:  Le langage des sch\'emas},
Publications math\'ematiques de l'I.H.\'E.S. {\bf no 4} (1960).

\bibitem[EGAIV]{EGAIV} A. Grothendieck (avec la collaboration de J.
Dieudonn\'e), {\it El\'ements de G\'eom\'etrie Alg\'ebrique IV: \'Etude locale des sch\'emas et des morphismes de sch\'emas},
Publications math\'ematiques de l'I.H.\'E.S. {\bf no 20, 24, 28} and {\bf 32}
(1964-1967).



\bibitem[FP]{FP} R. Fedorov, I. Panin, {\it 
A proof of Grothendieck-Serre conjecture on principal bundles over regular local rings containing infinite fields},
to appear in Publications Math\'ematiques de l'I.H.\'E.S.





\bibitem[GP1]{GP1} P. Gille and A. Pianzola, {\it Galois cohomology and 
forms of algebras over Laurent polynomial rings},  Mathematische Annalen  
{\bf 338} (2007), 497--543.

\bibitem[GP2]{GP2} P. Gille and A. Pianzola, {\it Isotriviality and
\'etale cohomology of  Laurent polynomial rings},   J. Pure Appl. Algebra  {\bf 212}  (2008), 780--800.

\bibitem[GP3]{GP3} P. Gille and A. Pianzola, {\it  Torsors, reductive group schemes and extended affine Lie algebras},
Memoirs of  AMS {\bf  1063} (2013). 


\bibitem[Gi]{Gi} J. Giraud, {\it Cohomologie non ab\'elienne},  Die Grundlehren der
mathematischen Wissenschaften {\bf 179} (1971), Springer-Verlag.

\bibitem[Gr]{Gr} A.  Grothendieck, {\it Le groupe de Brauer. II},  Th\'eorie cohomologique (1968),  Dix Expos\'es
 sur la Cohomologie des Sch\'emas,  67--87, North-Holland.






\bibitem[KLP]{KLP} V. Kac, M. Lau and A. Pianzola. {\it Differential conformal
superalgebras and their forms,}  Advances in Mathematics 
{\bf 222} (2009), 809--861.


\bibitem[M]{M} J.\,S. Milne, { \it Etale cohomology}, Princeton University Press (1980).


\bibitem[N]{N} Y. A. Nisnevich, {\it Espaces homog\`enes principaux rationnellement triviaux et arithm\'etique
 des sch\'emas en groupes r\'eductifs sur les anneaux de Dedekind}, C. R. Acad. Sci. Paris S\'er. I Math. {\bf  299}
  (1984),   5--8.




 
\bibitem[PSV]{PSV} I.  Panin, A. Stavrova and N. Vavilov, {\it 
On Grothendieck-Serre’s conjecture concerning principal G-bundles over reductive group schemes: I},
 Compositio Math. {\bf  151} (2015),  535--567. 

\bibitem[PZ]{PZ} G. Pappas, X. Zhu, {\it Local models of Shimura varieties and a conjecture of Kottwitz}, Inventiones Mathematicae {\bf 194} (2013), 147--254.
 
 
\bibitem[P]{P} R. Parimala, {\it Quadratic spaces over Laurent extensions 
of Dedekind domains},
Trans. Amer. Math. Soc. {\bf 277} (1983), 569--578.







\bibitem[R]{R} M.\,S. Raghunathan, {\it  Principal bundles on affine space and bundles on the projective line},
 Math. Ann. {\bf 285} (1989), 309–-332. 










\bibitem[Se]{Se1} J.-P. Serre, {\it  Espaces fibr\'es alg\'ebriques},  S\'eminaire Claude Chevalley {\bf  3} (1958), Expos\'e no. 1, 37 p. 




\bibitem[SGA1]{SGA1} {\it S\'eminaire de G\'eom\'etrie alg\'ebrique de
l'I.H.\'E.S.,  Rev\^etements \'etales et groupe fondamental, dirig\'e
par  A. Grothendieck},  Lecture Notes in Math. {\bf 224} (1971), Springer.

\bibitem[SGA3]{SGA3} {\it S\'eminaire de G\'eom\'etrie alg\'ebrique de
l'I.H.\'E.S., 1963-1964, sch\'emas en groupes, dirig\'e par M.
Demazure et A. Grothendieck},  Lecture Notes in Math. {\bf 151-153} (1970),
Springer.



\bibitem[St]{St} A. Steinmetz-Zikesch, {\it Alg\`ebres de Lie de dimension infinie et th\'eorie de la descente},
M\'emoire de la SMF {\bf 129} (2012).






\end{thebibliography}
\end{document}